\documentclass[a4paper, article,  11pt, oneside]{memoir}

\usepackage[utf8]{inputenc}

\usepackage[backend=bibtex, firstinits=true, url = false, doi=false, isbn=false, url=false, style=alphabetic]{biblatex}
\bibliography{refs.bib}

\usepackage[T1]{fontenc}
\usepackage{textcomp}
\usepackage{xfrac}
\usepackage{mathtools}

\usepackage{tikz}

\usepackage[scaled=0.85]{berasans}
\usepackage[scaled=0.84]{beramono}

\usepackage{booktabs}

\usepackage{titletoc}
\usepackage{titlesec}
\titleformat{\chapter}[display]
  {\normalfont\sffamily\huge\bfseries}
  {\chaptertitlename\ \thechapter}{20pt}{\Huge}
\titleformat{\section}
  {\normalfont\sffamily\Large\bfseries}
  {\thesection}{1em}{}
	\titleformat{\subsection}
  {\normalfont\sffamily\large\bfseries}
  {\thesubsection}{1em}{}

\counterwithout{section}{chapter}
\titleformat*{\paragraph}{\bfseries\sffamily}

\usepackage{amssymb, amsfonts, mathrsfs}
 \numberwithin{equation}{section}
\usepackage[amsmath,thmmarks]{ntheorem}
\usepackage[euler-hat-accent, euler-digits]{eulervm}
\usepackage{euscript}
\usepackage[english]{babel}

\usepackage{bbm}
\usepackage{enumitem}
\usepackage{nameref}
\usepackage{hyperref}
\usepackage{multirow}
\usepackage{booktabs}

\usepackage{lastpage, setspace}

\usepackage{graphicx, pst-all}
\usepackage[all]{xy}

\usepackage{tikz}
\usetikzlibrary{matrix}

\usepackage{multicol}
\usepackage{listings}
\usepackage{verbatim}
\usepackage{pdflscape}
\usepackage{color}
\usepackage{bm}

\usepackage{bbm}
\usepackage{setspace}

\DeclareMathOperator{\sgn}{sgn}
\DeclareMathOperator{\Trr}{tr}

\DeclareMathOperator\bl{bl}

\DeclareMathOperator\Stab{Stab}

\DeclareMathOperator\Diag{diag}

\DeclareMathOperator\rk{rk}

\newcommand{\N}{\mathbb{N}}

\newcommand{\R}{\mathbb{R}}

\newcommand\norm[1]{\left\lVert#1\right\rVert}

\newcommand\abs[1]{\left\lvert#1\right\rvert}

\renewcommand{\epsilon}{\ensuremath\varepsilon}

\renewcommand{\phi}{\ensuremath{\varphi}}

\newtheorem{theorem}{\sffamily{Theorem}}[section]
\newtheorem{question}{\sffamily{Question}}[section]

\newtheorem{definition}[theorem]{\sffamily{Definition}}
\newtheorem{corollary}[theorem]{\sffamily{Corollary}}
\newtheorem{lemma}[theorem]{\sffamily{Lemma}}

\newtheorem{remark}[theorem]{\sffamily{Remark }}

\theoremstyle{nonumberplain}
\theorembodyfont{\normalfont}
\theoremsymbol{\ensuremath{\blacksquare}}
\newtheorem{proof}{\sffamily{Proof}}

\newtheorem{profff}{\sffamily{Proof of Theorem \ref{thm:finite}}}
\newtheorem{proffff}{\sffamily{Proof of Corollary \ref{cor:translation}}}

\newtheorem{FirstAidKit}{\sffamily{Proof of Theorem \ref{thm:polyhedral}}}

\makeatletter
\newcommand{\displaybump}{\hbox to \@totalleftmargin{\hfil}}
\makeatother

\psset{linewidth=0.4pt}
\psset{arrowsize=3pt}

\SelectTips{cm}{}
\newdir{ >}{{}*!/-5pt/\dir{>}}

\newcommand*{\TitleFont}{
      \usefont{\encodingdefault}{\sfdefault}{b}{n}
      \fontsize{18}{20}
      \selectfont}

\begin{document}
\title{\TitleFont Computation of maximal projection constants}
\author{Giuliano Basso}


%

\maketitle
\begin{abstract}
The linear projection constant \(\Pi(E)\) of a finite-dimensional real Banach space \(E\)
is the smallest number \(C\in [0,+\infty)\) such that \(E\) is a \(C\)-absolute retract in the category of real Banach spaces with bounded linear maps. 
We denote by \(\Pi_n\) the maximal linear projection constant amongst 
\(n\)-dimensional Banach spaces. In this article, we prove that \(\Pi_n\) may be determined by computing eigenvalues of certain two-graphs. 
From this result we obtain that the relative projection constants of codimension \(n\) converge to \(1+\Pi_n\). Furthermore, using the classification of \(K_4\)-free two-graphs, we give an alternative proof of \(\Pi_2=\frac{4}{3}\).
We also show by means of elementary functional analysis that for each integer \(n\geq 1\) there exists a
polyhedral \(n\)-dimensional Banach space \(F_n\) such that \(\Pi(F_n)=\Pi_n\). 

\end{abstract}

\setcounter{section}{0}
\section{Introduction}

\paragraph{\sffamily Overview}
As a consequence of ideas developed by Lindenstrauss, cf. \cite{lindenstrauss1964}, for a finite-dimensional Banach space \(E\subset \ell_\infty(\N)\) the smallest constant \(C\in[0,+\infty)\) such that \(E\) is an absolute \(C\)-Lipschitz retract is completely determined by the linear theory of \(E\).
Indeed, Rieffel, cf. \cite{rieffel2006lipschitz}, established that it is equal to the \textit{linear projection constant} of \(E\), which is the number \(\Pi(E)\in[0,+\infty]\) defined as
\begin{equation*}
\inf \left\{ \norm{P}  \, \mid \,   P\colon \ell^\infty(\N) \to E \textrm{ bounded surjective linear map with } P^2=P \right\}.
\end{equation*}


Linear projections have been the object of study of many researchers and the literature can be traced back to the classical book by Banach, cf. \cite[p.244-245]{banach1932theorie}. The question about the maximal value \(\Pi_n\) of the linear projection constants of \(n\)-dimensional Banach spaces has persisted and is a notoriously difficult one. 
In this article, we establish a formula that relates \(\Pi_n\) with eigenvalues of certain two-graphs. This reduces the problem (in principle) to the classification of certain two-graphs and thus
allows the introduction of tools from graph theory. 
Following this approach, we present an alternative proof of \(\Pi_2=\frac{4}{3}\), see \textsc{\ref{par:determin}},  and we establish that that the relative projection constants of codimension \(n\)  converge to \(1+\Pi_n\), see Corollary \ref{cor:translation}. In the remainder of this overview, we summarize the current state of the theory.

For \(n \geq 1\), define \(\mathsf{Ban}_n\) to be the set of linear isometry classes of \(n\)-dimensional Banach spaces over the real numbers. 
The set \(\mathsf{Ban}_n\) equipped with the Banach-Mazur distance is a compact metric space, cf. \cite{tomczak1989banach}. Thus, the map \(\log \circ \, \Pi\colon \mathsf{Ban}_n \to [0,+\infty) \) is 
1-Lipschitz and consequently for all \(n\geq 1\) the maximal projection constant of order \(n\),
\[\Pi_n:=\max\big\{ \Pi(X) :X\in \mathsf{Ban}_n \big\},\]
is a well-defined real number. Apart form \(\Pi_1=1\), the only known value is \(\Pi_2=\frac{4}{3}\), due to Chalmers and Lewicki, cf. \cite{BruceL2010}. 
There is numerical evidence indicating that \(\Pi_3=(1+\sqrt{5})/2\), cf. \cite[Appendix B]{foucart2017maximal}, but to the author's knowledge, there is no known candidate for \(\Pi_n\) for all \(n \geq 4\). From a result of Kadets and Snobar, cf. \cite{kadets1971certain}, 
\[\Pi_n \leq \sqrt{n}.\]
Moreover, K\"onig, cf. \cite{Konig1985}, has shown that this estimate is asymptotically the
best possible. Indeed, there exists a sequence \((X_{n_k})_{k\geq 1}\) of finite-dimensional real Banach spaces such that  \(\dim(X_{n_k})=n_k\), where \(n_k\to +\infty\) for \(k\to +\infty\), and  
\[\lim_{k \to +\infty} \, \frac{\Pi(X_{n_k})}{\sqrt{n_k}} = 1.\]
There are many non-isometric maximizers of the function \(\Pi_n(\cdot)\), cf. \cite{HermannKoenig2003}. 
A finite-dimensional Banach space is called \textit{polyhedral} if its unit ball is a polytope. 
Equivalently, a finite-dimensional Banach space \((E, \norm{\cdot})\) is polyhedral if there exists an integer \(d\geq 1\) such that 
\((E, \norm{\cdot})\) admits a linear isometric embedding into \(\ell_\infty^d\). 
Using a result of Klee, cf. \cite[Proposition 4.7]{klee1960}, and elementary functional analysis, we show that there exist maximizers of \(\Pi_n(\cdot)\) that are polyhedral, see Theorem \ref{thm:polyhedral}. 

In the 1960s, Gr\"unbaum, cf. \cite{grunbaum1960projection}, calculated \(\Pi(\ell^n_1)\), \(\Pi(\ell^n_2)\) and \(\Pi(X_{_{hex}})\), where \(X_{_{hex}}\) is the 2-plane with the hexagonal norm.
In particular, \(\Pi(X_{_{hex}})=\frac{4}{3}\), which Gr\"unbaum conjectured to be the maximal value of \(\Pi(\cdot)\) amongst \(2\)-dimensional Banach spaces. In 2010, Chalmers and Lewicki presented an intricate proof of Gr\"unbaum's conjecture employing the implicit function theorem and Lagrange multipliers, cf. \cite{BruceL2010}. 


Our main result, see Theorem \ref{thm:finite}, provides a characterization of the number \(\Pi_n\) in terms of certain maximal sums of eigenvalues of two-graphs that are \(K_{n+2}\)-free. In \cite{frankl1984exact},  Frankl and F{\"u}redi give a full description of two-graphs that are \(K_{4}\)-free. Via this description and Theorem \ref{thm:finite} we can derive from first principles that \(\Pi_2=\frac{4}{3}\).
This is done in Section \ref{sec:applo}.

Next, we introduce the necessary notions from the theory of two-graphs that are needed to properly state our main result.

\paragraph{\sffamily Two-graphs} The subsequent definition of a two-graph via cohomology follows Taylor \cite{taylor1977regular}, and Higman \cite{higman1973remark}; see also \cite[Remark 4.10]{SEIDEL1991146}. Let \(V\) denote a finite set. For each integer \(n\geq 0\) we set \[E_n(V):=\big\{ B\subset V : \abs{B}=n \big\} \textrm{ and  }\mathcal{E}_n(V):=\big\{ f\colon E_n(V)\to \mathbbm{F}_2\big\},\]
where \(\mathbbm{F}_2\) denotes the field with two elements. Elements of \(\mathcal{E}_2(V)\) are finite simple graphs. If \(n\) is strictly greater than the cardinality of \(V\), then \(\mathcal{E}_n(V)\) consists only of the empty function \(\varnothing \to \mathbbm{F}_2\). For each \(f\in \mathcal{E}_n(V)\) the map \(\delta f\in\mathcal{E}_{n+1}(V)\) is given by
\[B\mapsto \sum_{v\in B} f(B\setminus \{v\}). \]
Clearly, it holds that \(\delta\circ \delta=0\), where \(0\) denotes the neutral element of the group \(\mathcal{E}_{n+2}(V)\). Two-graphs can be defined as follows.
\begin{definition}[\sffamily{two-graph}]
A two-graph is a tuple \(T=(V,\Delta)\), where \(V\) and \(\Delta\) are finite sets and there exists a map \(f_T\in \mathcal{E}_3(V)\) such that \(\delta f_T=0\) and \(\Delta=f_T^{-1}(1)\). 
The cardinality of \(V\) is called the order of \(T\). 
\end{definition}
Among other things, two-graphs naturally  occur in the study of systems of equiangular lines and 2-transitive permutation groups;
authoritative surveys are \cite{SEIDEL1991146, SEIDEL1992297}. Given a two-graph \(T=(V, \Delta)\), the following set is always non-empty: 
\[[T]:=\big\{f\colon E_2(V)\to \mathbb{F}_2 : \delta f=f_T\big\}.\]
Each \(f\in [T]\) gives rise to a graph \(G_f:=(V, f^{-1}(1))\). 
The \textit{Seidel adjacency matrix} of a graph \(G=(V,E)\)  is the matrix \(S(G)\), which is the symmetric \(|V|\times |V|\)-matrix given by
\[ S(G)_{ij}=
\begin{cases}
0 & \textrm{ if } i=j \\
-1  & \textrm{ if \(i\) and \(j\) are adjacent } \\
1  & \textrm{ otherwise. } \\
\end{cases}
\]

For each choice \(f_1,f_2\in [T]\) the matrices \(S(G_{f_1})\) and \(S(G_{f_2})\) have the same spectrum. 

By definition, the \textit{eigenvalues of \(T=(V, \Delta)\)} are the real numbers \[\lambda_1(T)\geq \ldots \geq \lambda_{\abs{V}}(T)\] that are the eigenvalues of \(S(G_{f})\) for \(f\in [T]\) (counted with multiplicity).
This definition is independent of \(f\in [T]\). 

We say that a two-graph \(T=(V,\Delta)\) is \(K_n\)-free if there is no injective map \(\phi\colon \{1, \ldots, n\} \to V\) such that 
\(\big\{\phi(v_1), \phi(v_2), \phi(v_3)\big\}\in \Delta\) for all distinct points \(v_1,v_2, v_3\in V_n\). 

\paragraph{\sffamily Main result}

Our main result reads as follows:
 \begin{theorem}\label{thm:finite}
If \(n\geq 1\) is an integer, then
\[\Pi_n=\sup_{d\geq 1} \,\max\left\{ \frac{n}{d}+\frac{1}{d} \sum_{k=1}^n \lambda_k(T) \, : \, T \textrm{ is a \(K_{n+2}\)-free two-graph of order \(d\)} \right\}. \]
\end{theorem}

To prove Theorem \ref{thm:finite}, we invoke a simple trick, see Lemma \ref{lem:holyGrail}, that allows us to greatly narrow
down the matrices that need to be considered. This is done in Section \ref{sec:formula}.

\paragraph{\sffamily Relative projection constants}

The following question has first been systematically addressed by K\"onig, Lewis, and Lin in \cite{konig1983finite}:

\begin{question}\label{quest:gruen}
Let \(n, d\geq 0\) be integers. What is 
\[\Pi(n,d):=\sup\big\{ \Pi(E) : E\subset \ell_d^\infty \textrm{ is an } n\textrm{-dimensional Banach space}\big\}?\]
\end{question} 
By definition, \(\sup \varnothing =-\infty\). Clearly, \(\Pi(d,d)=1\) and it is a direct consequence of the classical Hahn-Banach theorem that \(\Pi(1,d)=1\) for all integers \(d\geq 1\).
The quantity \(\Pi(d-1,d)\) has been examined by Bohnenblust, cf. \cite{bohnenblust1938convex}, where it is shown that \(\Pi(d-1,d)\leq 2 -\frac{2}{d}\). 
In \cite{CHALMERS2009553}, Chalmers and Lewicki determined the exact value of \(\Pi(3,5)\). In \cite{konig1983finite}, König, Lewis, and Lin established the general upper bound
\[\Pi(n,d)\leq  \frac{n}{d}+\sqrt{ \big(d-1\big)\frac{n}{d}\big(1-\frac{n}{d}\big)}\]
with equality if and only if \(\R^n\) admits a system of \(d\) distinct equiangular lines.
Thereby, as \(\R^3\) admits a system of six equiangular lines, cf. \cite[p. 496]{LEMMENS1973494}, it holds that
\[\Pi(3,6)=\frac{1+\sqrt{5}}{2}.\]
In light of
\[\Pi(4,6)=\frac{5}{3},\]
which we demonstrate in Paragraph \ref{para:graph}, up to \(d=6\) all exact values of \(\Pi(n ,d)\) for \(1\leq n \leq d\) are now computed.  It is well-known that 
\[\Pi(n,d) \leq \Pi(n,d+1) \,\textrm{  and  } \,\Pi(n,d)\leq \Pi(n+1,d+1)\]
 for all \(1\leq n \leq d\), cf. \cite{CHALMERS2009553}. Via Theorem \ref{thm:finite},
we infer the following asymptotic relation between these two increasing sequences: 
\begin{corollary}\label{cor:translation}
For each integer \(n\geq 1\) we have
\[1+\Pi_n=\lim_{d\to+\infty} \Pi\left(d-n,d\right).\]
\end{corollary}
A proof of Corollary \ref{cor:translation} is given in Section \ref{sec:formula}. If \(n=1\), then
Corollary \ref{cor:translation} follows directly from the fact that Bohnenblust's upper bound of \(\Pi(d-1,d)\) is sharp, cf. \cite[Lemma 2.6]{CHALMERS2009553}. 
Recently, the special case \(n=2\) has been considered by Sokołowski in \cite{doi:10.1080/01630563.2016.1266655}.

Recall that \(\Pi(1,d)=\Pi(1,1)=1\) for all \(d\geq 1\).
The proof of  Gr\"unbaum's conjecture, cf. \cite{BruceL2010}, shows that \[\Pi(2,d)=\Pi(2,3)=\frac{4}{3}\,\,\,\, \textrm{ for all } d\geq 3.\]
Numerical experiments, cf. \cite[Appendix B]{foucart2017maximal}, suggest that if \(d\in \{6, \ldots, 10\}\), then \(\Pi(3,d)=\Pi(3,6)\). 
Since \(\Pi_n(\cdot)\) admits a polyhedral maximizer, the sequence \(\Pi(n, \cdot)\) stabilizes eventually.

\begin{theorem}\label{thm:polyhedral}
Let \(n\geq 1\) be an integer. There exists a polyhedral \(n\)-dimensional Banach space \((F_n, \norm{\cdot})\) such that 
\[\Pi(F_n)=\Pi_n.\]
As a result, there is an integer \(D\geq 1\) such that
\[\Pi(n,d)=\Pi(n,D)\]
for all \(d\geq D\).
\end{theorem}
A proof of Theorem \ref{thm:polyhedral} can be found in Section \ref{sec:polyball}.


\section{A formula for \(\Pi_n\)}\label{sec:formula}
\paragraph{\sffamily A result of Chalmers and Lewicki}
Let \(d\geq 1\) be an integer and set
\[\mathcal{A}_d:=\big\{ \mathbbm{1}_d+S : S \textrm{ is a Seidel adjacency matrix of a simple graph of order } d \big\}.\]
Moreover, we use \(\mathcal{D}_d\) to denote the set of all diagonal \(d\times d\)-matrices  that have trace equal to one and whose diagonal entries are non-negative. 

For \(A\in \mathcal{A}_d\) and \(D\in \mathcal{D}_d\) we write \(\lambda_1(\sqrt{D} A \sqrt{D})\geq \ldots \geq \lambda_d(\sqrt{D} A \sqrt{D})\) for the eigenvalues of the symmetric matrix \(\sqrt{D} A \sqrt{D}\) (counted with multiplicity). 
The subsequent result, due to Chalmers and Lewicki, characterizes the values \(\Pi(n,d)\) in terms of maximal sums of eigenvalues of matrices of the form \(\sqrt{D}A\sqrt{D}\).

\begin{theorem}[Theorem 2.3 in \cite{BruceL2010}]\label{thm:known}
Let \(1\leq n \leq d\) be integers.  The value \(\Pi(n,d)\) is attained and equals
\begin{equation*}
\begin{split}
&\max\left\{ \sum_{k=1}^n \lambda_k\left(\sqrt{D}A\sqrt{D}\right) : A\in \mathcal{A}_d  \textrm{ \normalfont and } D\in \mathcal{D}_d  \right\}.
\end{split}
\end{equation*}
\end{theorem}

\paragraph{\sffamily Blow-up of matrices}
Let \(i\geq 1\) be an integer and consider the map
\[\bl_i \colon \bigcup\limits_{d\geq i} \mathcal{A}_d \to \bigcup\limits_{d\geq i+1} \mathcal{A}_d,  \quad
A \mapsto \bl_i(A):=
\left[
\begin{array}{cc}
A & a_i^t \\

a_i & 1,
\end{array}
\right]
\]
where \(a_i\) denotes the \(i\)-th row of \(A\). By construction, the \(i\)-th row of \(\bl_i(A)\) and the last row of \(\bl_i(A)\) coincide.
We say that the matrix \(\bl_i(A)\) is a \textit{blow-up} of \(A\) (with respect to the \(i\)-th row).

If \(A\in \mathcal{A}_d\) is a matrix and \(D\in\mathcal{D}_d\) is positive-definite, then all eigenvalues of \(AD\) are real, for \(AD\) is equivalent to the symmetric matrix \(\sqrt{D}A\sqrt{D}\). With a similar argument, one can show that even if \(D\) is positive-semidefinite, then all eigenvalues of \(AD\) are real. We use the notation \[\lambda(AD):=(\lambda_1(AD), \ldots, \lambda_d(AD)),\]
where \(\lambda_1(AD)\geq \ldots \geq \lambda_d(AD)\) are the eigenvalues of \(AD\) (counted with multiplicity). 
The lemma below is the key step in the proof of Theorem \ref{thm:finite}. 

\begin{lemma}\label{lem:holyGrail}
Let \(A^\prime\in \mathcal{A}_{d-1}\) be a matrix, let \(A:=\normalfont{\bl}_i(A^\prime)\) for some integer \(1\leq i \leq d-1\) and let \(D:=\Diag(d_1, \ldots, d_d)\in\mathcal{D}_d\) be an invertible matrix. We set \(D^\prime:=\Diag(d_1, \ldots, d_{i-1}, d_{i}+d_{d}, d_{i+1}, \ldots, d_{d-1})\). Then \(D^\prime \in \mathcal{D}_{d-1}\) is invertible, \(\lambda(AD)\) has a zero entry and 
\[\lambda(A^\prime D^\prime) \textrm{ is obtained from } \lambda(AD)  \textrm{ by deleting a zero entry }.\]
\end{lemma}
\begin{proof}
For each integer \(1\leq k \leq d\) let \(s_k\) denote the \(k\)-th row of \(A\). By assumption, 
\[s_d=s_i.\]
Let \(\lambda\) be an eigenvalue of \(A^\prime D^\prime\) and let \(x^\prime:=(x_1, \ldots, x_{d-1})\in \R^{d-1}\) be a corresponding eigenvector.
We define \(x:=(x_1, \ldots, x_{d-1},  x_i)\).
For all \(1 \leq k <d\) we compute
\begin{equation}\label{eq:meat}
\begin{split}
&\langle  D s_k ,  x\rangle_{_{\R^d}}= s_{ki}d_i x_i+s_{kd}d_d  x_i+\sum_{\ell\neq d,i}^d s_{k\ell} d_\ell x_\ell \\
&=s_{ki}d_i x_i+s_{ki} d_d  x_i+\sum_{\ell\neq d,i}^d s_{k\ell} d_\ell x_\ell = \langle  D^\prime s_k^{\prime} ,  x^\prime \rangle_{_{\R^{d-1}}}.
\end{split}
\end{equation}
Thus, for all \(1 \leq k <d\) we have
\[\langle  D s_k ,  x\rangle_{_{\R^d}}=\langle  D^\prime s_k^{\prime} ,  x^\prime \rangle_{_{\R^{d-1}}}=\lambda x_{k}.\]
Furthermore,
\[\langle  D s_d ,  x\rangle_{_{\R^d}}= \langle  D s_i ,  x\rangle_{_{\R^d}}=\lambda  x_i;\]
as a result, the vector \(x\) is an eigenvector of \(AD\) with corresponding eigenvalue \(\lambda\). 

Next, we show that \(AD\) and \(A^\prime D^\prime\) have the same rank. There exists a principal submatrix \(T\) of \(A\) 
such that \(T\) is invertible and \(\rk(A)=\rk(T)\). This is well-known, cf. for example \cite[Theorem 5]{thompson1968principal}.
Clearly, \(T\) cannot be obtained from \(A\) by keeping the \(i\)-th and \(d\)-th column simultaneously; thus, \(T\) is also a principal submatrix of \(A^\prime\).
Therefore,
\[\rk(A^\prime)\leq\rk(A)=\rk(T)\leq \rk(A^\prime)\]
and thereby \(\rk(A)=\rk(A^\prime)\). Now, via Sylvester's law of interia
\[\rk(AD)=\rk(\sqrt{D}A\sqrt{D})=\rk(A)=\rk(A^\prime)=\rk(A^\prime D^\prime),\]
as claimed. To summarize, \(AD\) and \(A^\prime D^\prime\) have the same rank and if \(\lambda\) is an eigenvalue of \(A^\prime D^\prime\), then \(\lambda\) is an
eigenvalue of \(AD\). This completes the proof.
\end{proof}

\paragraph{\sffamily Proofs of the main results }
Now, we have everything at hand to verify Theorem \ref{thm:finite}. 
\begin{profff}

We set
\[\Phi_n:=\sup_{d\geq 1} \,\max\left\{ \frac{1}{d} \sum_{k=1}^n \lambda_k(A) : A\in \mathcal{A}_d \right\}.\]
First, we show for all \(d\geq n\) that
\[\Pi(n,d)\leq \Phi_n.\]
We abbreviate
\[\pi_n(AD):=\sum_{k=1}^n \lambda_k(AD).\]
Due to Theorem \ref{thm:known}, there exist matrices \(A\in \mathcal{A}_d\) and \(D\in\mathcal{D}_d\) such that 
\[\Pi(n,d)=\pi_n(AD).\]
Choose a sequence \(D_k\in \mathcal{D}_d\) of invertible matrices with rational entries satisfying
\begin{equation}\label{eq:approxU}
\Pi(n,d)\leq \pi_n(AD_k)+\frac{1}{2^k}.
\end{equation}
This is possible since \(\pi_n(AD)=\pi_n(\sqrt{D} A \sqrt{D})\) and because the map \(\pi_n(\cdot)\) is continuous on the set of symmetric matrices, cf. \cite[p. 44]{overton1992sum}. 
Fix \(k\geq 1\). By finding a common denominator, we may write 
\[D_k=\frac{1}{m}\Diag(n_1,\ldots, n_d),\]
where \(n_i\geq 1\) for all \(1 \leq i \leq d\) and \(m=n_1+\dotsm+n_d\).  We set
\[A_k:=\bl_{d}^{(n_d-1)}(\dotsm(\bl_1^{(n_{1}-1)}(A))\dotsm),\]
where we use the convention \(\bl_i^0(A)=A\). 
Note that \(A_{k}\in\mathcal{A}_{m}\). By applying Lemma \ref{lem:holyGrail} repeatedly, we get
that \(\lambda(A D_k)\) is obtained from  \(\lambda\left(A_k \frac{1}{m}\mathbbm{1}_m\right)\) by deleting exactly \((m-d)\) zero entries. 
As a result,
\begin{equation}\label{eq:upperPhi}
\pi_n(AD_k)\leq \frac{\pi_n(A_k)}{m}\leq \Phi_n.
\end{equation}
Thus, by combining \eqref{eq:upperPhi} with \eqref{eq:approxU}, we obtain
\[\Pi(n,d)\leq \Phi_n.\]
It is well-known that
\[\Pi_n=\lim_{d\to +\infty} \Pi(n,d).\]
Hence,
\[\Pi_n\leq \Phi_n.\]
The inequality \(\Phi_n\leq \Pi_n\) is a direct consequence of Theorem \ref{thm:known}. Putting everything together, we conclude
\[\Pi_n=\Phi_n.\]

We are left to show that it suffices to consider \(K_{n+2}\)-free two-graphs. 
To this end, fix an integer \(d> n\) and let \(A\in \mathcal{A}_d\) be a matrix such that
\begin{equation*}
\pi_n(A)=\max\big\{ \pi_n(A^\prime) : A^\prime\in \mathcal{A}_d\big\}. 
\end{equation*}
As the symmetric matrix \(A\) is orthogonally diagonalizable, there
are orthonormal vectors \(u_1,\ldots ,u_n \in \R^d\) such that
\[\pi_n(A)=\Trr(A UU^t),\]
where \(U\) is the matrix that has the vectors \(u_i\) as columns. Let \(r_k\) for \(1\leq k \leq d\) be the rows of the matrix \(U\). 
We use \(e_1, \ldots, e_d\in \mathbb{R}^d\) to denote the standard basis.  Fix \(1 \leq i, j \leq d\) and let \(\epsilon\in \mathbb{R}\) be a real number. We set
\begin{equation*}
A(i,j; \epsilon):=
\begin{cases}
\epsilon \, \sgn\big(\langle r_i,r_j \rangle_{_{\R^n}}\big) e_ie_j^t & \textrm{ if } \langle r_i,r_j \rangle\neq 0 \\
\epsilon \, e_ie_j^t & \textrm{ otherwise } 
\end{cases}
\end{equation*} 
and
\begin{equation*}
\widehat{A}_{\epsilon}:=A+\frac{1}{2}\big(A(i,j;\epsilon)+A(j,i;\epsilon)\big).
\end{equation*}
Clearly, \(\widehat{A}_{\epsilon}\) is symmetric.
Hereafter, we show that \(\langle r_i, r_j \rangle \neq 0\). To this end, suppose that \(\langle r_i, r_j \rangle= 0.\)

We set \(\epsilon_\star:= -4\,\sgn(a_{ij})\), and we observe that \(\widehat{A}_{\epsilon_\star}\in\mathcal{A}_d\). Further, we abbreviate \(\widehat{A}:=\widehat{A}_{\epsilon_\star}\). It holds that
\begin{equation}\label{eq:bocer}
\pi_n(A)=\Trr\left( A U U^t \right)=\Trr\left( \widehat{A} U U^t \right)-\epsilon_\star \sgn\big(\langle r_i,r_j \rangle\big) \langle r_i,r_j \rangle.
\end{equation}
 Via von Neumann's trace inequality, cf. \cite{Mirsky1975}, we obtain
 \[\Trr\left( \widehat{A} U U^t \right) \leq \pi_n(\widehat{A})\leq \pi_n(A);\]
thus, 
\[\Trr\left( \widehat{A} U U^t \right)=\pi_n(\widehat{A})=\pi_n(A).\]
The equality case of von Neumann's trace inequality occurs. Therefore, the diagonalizable matrices \(U U^t\) and \( \widehat{A}\) are simultaneously orthogonally diagonalizable and thereby commute. This implies that \(UU^t\) and \(\tfrac{1}{2}\big(A(i,j;\epsilon_\star)+A(j,i;\epsilon_\star)\big)\) commute; as a result, we get
that 
\begin{equation*}
\begin{split}
&\langle r_i,r_i\rangle=\langle r_j, r_j\rangle, \\
&\langle r_i,r_k \rangle= 0, \quad \quad \textrm{ for all } k\neq i \textrm{ with } k\in \{1, \ldots d\}, \\
&\langle r_j,r_k \rangle= 0, \quad \quad \textrm{ for all } k\neq j \textrm{ with }  k\in \{1, \ldots d\}. 
\end{split}
\end{equation*}

By applying the same argument to \(\langle r_i,r_k \rangle= 0\) for every \(k\neq i, k\in \{1, \ldots, d\}\), we may conclude that the vectors \(r_1, \ldots, r_d\in \R^n\) are orthogonal and none of them is equal to the zero vector. However, this is only possible if \(n=d\). Therefore, we have shown for \(d>n\) that \(\langle r_i, r_j \rangle \neq 0\) for all integers \(1 \leq i, j \leq d\). 

We claim that 
\begin{equation}\label{eq:important}
a_{ij}=\sgn\big(\langle r_i,r_j \rangle_{_{\R^{n}}}\big)
\end{equation}
for all \(1\leq i,j \leq d\). Because \(\langle r_i, r_j \rangle \neq 0\), this is a direct consequence of the maximality of \(\pi_n(A)\) and equality \eqref{eq:bocer}. 
Hence, we have shown that \(A\) and \(UU^t\) have the same sign pattern, which allows us to
invoke \cite[Lemma 2.1]{cia2013note}. From this result we see that \(A\) does not have a principal \((n+2)\times (n+2)\)-submatrix which has only \(-1\) as off-diagonal elements.
Such a matrix is the Seidel adjacency matrix of the complete graph on \(n+2\) vertices.
For that reason, we have shown that
\[\frac{1}{d}\pi_n(A)=\max\left\{ \frac{n}{d}+\frac{1}{d}\sum_{k=1}^n \lambda_k(T)  :  T \textrm{ is a \(K_{n+2}\)-free two-graph of order \(d\)} \right\}.\]
This completes the proof.
\end{profff}

We conclude this section with the proof of Corollary \ref{cor:translation}.
\begin{proffff}
Let \(J_2\in \mathcal{A}_2\) denote the all-ones matrix.
For every \(A\in \mathcal{A}_d\), the matrix \(A\otimes J_2\)
is contained in \(\mathcal{A}_{2d}\), where \(\otimes\) denotes the Kronecker product of matrices.
Moreover, since the eigenvalues of \(A\otimes J_2\) are precisely all possible products of an eigenvalue of \(A\) (counted with multiplicity) and an eigenvalue of \(J_2\) (counted with multiplicity), 
it is readily verified that
\[\frac{\pi_n(A)}{d}=\frac{\pi_n(A\otimes J_2)}{2d}.\]
Let \((\epsilon_\ell)_{\ell\geq 1}\) be a sequence of positive real numbers that converges to zero. Due to Theorem \ref{thm:finite} and the above, there exists a strictly increasing sequence \((d_\ell)_{\ell\geq 1}\) of integers and matrices \(A_\ell\in \mathcal{A}_{d_\ell}\) such that
\[ \Pi_n\leq \frac{\pi_n(A_\ell)}{d_\ell}+\epsilon_\ell.\]
We have
\[\pi_n(A_\ell)=d_\ell-\sum_{k=n+1}^{d_\ell} \lambda_k(A_\ell)=d_\ell+\sum_{k=1}^{d_{\ell}-n}\lambda_k(-A_\ell);\]
thus,
\[\pi_n(A_\ell)=d_\ell+\sum_{k=1}^{d_{\ell}-n}\lambda_k(\overline{A_\ell})-(d_\ell-n)2,\]
where \(\overline{A_\ell}=2\mathbbm{1}_{d_\ell}-A_\ell\). Consequently,
\[\Pi_n\leq \frac{2n}{d_\ell}-1+\frac{\pi_{d_\ell-n}(\overline{A_\ell})}{d_\ell}+\epsilon_\ell.\]
Since \(\overline{A_\ell}\in \mathcal{A}_{d_\ell}\,\), we obtain
\[\Pi_n\leq \frac{2n}{d_\ell}-1+\Pi(d_\ell-n,d_\ell)+\epsilon_\ell.\]
Proposition 2 in \cite{foucart2017maximal} tells us that
\[\Pi(d-n,d)\leq \Pi_n+1\]
for all \(d\geq 1\).
Thus, 
\[\Pi_n\leq \frac{2n}{d_\ell}-1+\Pi(d_\ell-n,d)+\epsilon_\ell\leq \Pi_n+\frac{2n}{d_\ell}+\epsilon_\ell;\]
for that reason, the desired result follows. 
\end{proffff}


\section{Polyhedral maximizers of \(\Pi_n(\cdot)\)}\label{sec:polyball}

\paragraph{\sffamily Projections in \(E\) and \(E^\ast\)} Let \((E, \norm{\cdot})\) be a Banach space and let \(V\subset E\) and \(F\subset E^\ast\) denote linear subspaces. We
set
\[V^0:=\big\{\ell\in E^\ast \, : \, \ell(v)=0 \textrm{ for all } v\in V \big\}\subset E\] 
and
\[F_0:=\big\{x\in E \,: \, f(x)=0 \textrm{ for all } f\in F \big\}\subset E^\ast.\] 
Suppose that \(U\subset E\) is a linear subspace such that \(E=V\oplus U\). The map
\[P_V^U\colon E\to V,\,\quad v+u\mapsto v\]
is a linear projection onto \(V\). 
In the subsequent lemma we gather useful results from functional analysis.
\begin{lemma}\label{lem:FunkAnna}
Let \((E, \norm{\cdot})\) be a Banach space. 
\begin{enumerate}
\item If there exist closed linear subspaces \(V,U\subset E\) such that  \(V\) is finite-dimensional and \(E=V\oplus U\), then \(E^\ast=V^0\oplus U^0\), \(\dim(U^0)=\dim(V)\), 
\[(V^0)_0=V \textrm{ and } (U^0)_0=U.\]
\item If there exist closed linear subspaces \(F,G\subset E^\ast\) such that \(F\) is finite-dimensional and \(E^\ast=F\oplus G\), then \(E=F_0\oplus G_0\), \(\dim(G_0)=\dim(F)\), 
\[(F_0)^0=F \textrm{ and } (G_0)^0=G.\] 
\item  If there exist closed linear subspaces \(V,U\subset E\) such that \(V\) is finite-dimensional and \(E=V\oplus U\),  then
\begin{equation*}
\lVert P_V^U \rVert=\lVert P_{U^0}^{V^0} \rVert.
\end{equation*}

\end{enumerate}
\end{lemma}

\begin{proof}
We prove each item separately. 
\begin{enumerate}

\item If \(\ell \in V^0\cap U^0\), then \(\ell(x)=0\) for all \(x\in E\), implying \(V^0\cap U^0=\{ 0\}\).
As \(V\) and \(U\) are closed, we may deduce with the usual Hahn-Banach separation argument that
\[(V^0)_0=V \textrm{ and }  (U^0)_0=U.\]
Accordingly, \(((V^0)_0)^0=V^0\) and \(((U^0)_0)^0=U^0.\)
So, \(V^0\) and \(U^0\) are weak-star closed, cf. \cite[Theorem 4.7]{rudin1991functional}. 
The quotient \(E^\ast / V^0\) is a Hausdorff locally convex vector space and the quotient map \(\pi\colon E^\ast \to E^\ast / V^0\) is continuous.
We claim that the subspace \(\pi(U^0+V^0)\) is closed in \(E^\ast / V^0\). As every finite-dimensional linear subspace of a Hausdorff topological vector space is closed, cf. \cite[Theorem 1.21]{rudin1991functional},  it suffices to show that \(\pi(U^0+V^0)\) is finite-dimensional. Since
\[(E/ U)^\ast=U^0,\]
we see that \(U^0\) is finite-dimensional; thus,  the subspace \(\pi(U^0+V^0)\) is finite-dimensional as well. Therefore,  the subspace \(\pi(U^0+V^0)\) is closed in \(E^\ast / V^0\) and we get that \(V^0+U^0=\pi^{-1}(\pi(U^0+V^0))\) is weak-star closed.
Note that \(V^0+U^0\) is weak-star dense in \(E^\ast\), because \(V^0+U^0\) separates points of \(E\). 
This implies \(E^\ast=V^0\oplus U^0\), as desired.

\item Suppose that \(x\in F_0\cap G_0\). We have \(\ell(x)=0\) for all \(\ell\in E^\ast\). As the elements of \(E^\ast\) separate points, we
get that \(x=0\) and consequently \(F_0\cap G_0=\{0\}\). 
Since \(F\) is finite-dimensional and thereby weak-star closed, \cite[Theorem 4.7]{rudin1991functional} tells us that
\[(F_0)^0=F.\]
Using
\[(E/F_0)^\ast=(F_0)^0=F,\]
we may deduce that \(E/F_0\) is finite-dimensional. Let \(\pi\colon E\to E/F_0 \) denote the quotient map.
The linear subspace \(\pi(G_0+F_0)\subset E/F_0\) is finite-dimensional and thus closed. As a result,
\(F_0+G_0=\pi^{-1}(\pi(G_0+F_0))\) is closed. Via the familiar Hahn-Banach separation argument, it is not hard to check that
\(F_0+G_0\) is a dense subset of \(E\). For that reason, \(F_0\oplus G_0=E\). The first item tells us that
\[(F_0)^0\oplus (G_0)^0=E^\ast;\]
therefore \((G_0)^0=G\), as \((F_0)^0=F\). This completes the proof of the second item. 

\item We abbreviate \(F:=V^0\) and \(G:=U^0\). We compute
\begin{equation*}
\lVert P_{U^0}^{V^0} \rVert=\underset{\substack{\norm{f+g}=1,\\ f\in F,\, g\in G }}{\sup}\norm{g}=\underset{\substack{\norm{f+g}=1,\\ f\in F,\, g\in G }}{\sup} \,\,\, \underset{\substack{\norm{v+u}=1,\\ v\in V,\, u\in U }}{\sup} \abs{g(v)} 
\end{equation*}
and
\begin{equation*}
\underset{\substack{\norm{v+u}=1,\\ v\in V,\, u\in U }}{\sup}\,\,\,\underset{\substack{\norm{f+g}=1,\\ f\in F,\, g\in G }}{\sup}  \abs{g(v)}=\underset{\substack{\norm{v+u}=1,\\ v\in V,\, u\in U }}{\sup} \norm{v}=\lVert P_V^U \rVert,
\end{equation*}
as was to be shown. 
\end{enumerate}
\end{proof}

\paragraph{\sffamily Construction of polyhedral maximizers}
Let \((E, \norm{\cdot})\) be a Banach space and let \(F\subset E\) denote a finite-dimensional linear subspace.
The number
\[\Pi(F,E):=\inf\left\{ \norm{P} \, \mid \,   P\colon E \to F \textrm{ bounded surjective linear map with } P^2=P \right\}.\]
is called the \textit{relative projection constant} of \(F\) with respect to \(E\). 
The following theorem translates the calculation of relative projection constants to second preduals (if such a space exists).
\begin{theorem}\label{thm:secondPredual}
Let \((E, \norm{\cdot})\) be a Banach space and let \(F\subset E\) denote a finite-dimensional linear subspace.
If \((X, \norm{\cdot})\) is a Banach space such that \(E=X^{\ast\ast}\), then there exist a linear subspace \(V\subset X\) with \(\dim(V)=\dim(F)\) and
\[\Pi(F,E)=\Pi(V,X).\]
\end{theorem}
\begin{proof}
It is not hard to check that
\[\Pi(F,E):=\inf\big\{ \lVert P_F^G \rVert : E=F\oplus G, G\subset E \textrm{ closed linear subspace } \big\}. \]
We set \(V:=(F_0)_0\). On the one hand, using the second and third item of Lemma \ref{lem:FunkAnna}, we obtain
\[\Pi(V,X) \leq \Pi(F,E);\]
on the other hand, using the first and third item of Lemma \ref{lem:FunkAnna}, we infer
\[\Pi(F,E) \leq \Pi(V,X).\]
This completes the proof.
\end{proof}
We conclude this section with the proof of Theorem \ref{thm:polyhedral}.
\begin{FirstAidKit}
Let \(F\subset \ell_\infty\) be an \(n\)-dimensional linear subspace with
\[\Pi_n=\Pi(F, \ell_\infty).\]
Via Theorem \ref{thm:secondPredual}, there exists an \(n\)-dimensional linear subspace \(V\subset c_0\) such that
\[\Pi(F, \ell_\infty)=\Pi(V, c_0).\]
As \(\Pi(V, c_0)\leq \Pi(V)\), we get
\[\Pi_n=\Pi(V).\]
This completes the proof, since due to a result of Klee, cf. \cite[Proposition 4.7]{klee1960}, every finite-dimensional subspace of \(c_0\) is polyhedral.
\end{FirstAidKit}

\section{Applications: Computation of \(\Pi_2\) and \(\Pi(4,6)\) }\label{sec:applo}
\paragraph{\sffamily List of all \(K_4\)-free two-graphs}\label{para:classi}
Let \(n\geq 1\) be an integer, let \(\mathcal{R}_{2n+1}\subset \R^2\) be a regular \((2n+1)\)-gon centred at the origin and let \(V(\mathcal{R}_{2n+1})\) denote the vertices of
\(\mathcal{R}_{2n+1}\). Further, we let \(\mathcal{T}_{2n+1}\) denote the two-graph that has \(V(\mathcal{R}_{2n+1})\) as vertex set and \(\{v_1, v_2, v_3\}\subset V(\mathcal{R}_{2n+1})\) is an edge if and only if the origin is contained in the closed convex hull of \(v_1, v_2, v_3\).  It is readily verified that \(\delta(R_{2n+1}-\mathbbm{1}_{2n+1})=\mathcal{T}_{2n+1}\) for
\[R_{2n+1}:=\left( \begin{array}{ccc}
1 & j^t & -j^t \\
j & J_{n} & J_{n}-2L_n \\
-j &  J_{n}-2L_n^t &  J_{n}  \end{array} \right),\]
where \(j\in \R^n\) is the all-ones vector, \(J_n\) is the all-ones \(n\times n\)-matrix and \(L_n\) is the \(n\times n\)-matrix given by
\[(L_n)_{ij}:=
\begin{cases}
-1 & i>j \\
0 & \textrm{otherwise }.
\end{cases}
\]
Note that \(L_n\) has only \(-1\)'s below the diagonal and only \(0\)'s above the first sub-diagonal.
We set
\[A_6:=\left(
\begin{array}{cccccc}
 1 & 1 & 1 & 1 & 1 & 1 \\
 1 & 1 & 1 & 1 & -1 & -1 \\
 1 & 1 & 1 & -1 & 1 & -1 \\
 1 & 1 & -1 & 1 & -1 & 1 \\
 1 & -1 & 1 & -1 & 1 & 1 \\
 1 & -1 & -1 & 1 & 1 & 1 \\
\end{array}
\right).\]
\begin{figure}
\centering
\includegraphics[scale=0.3]{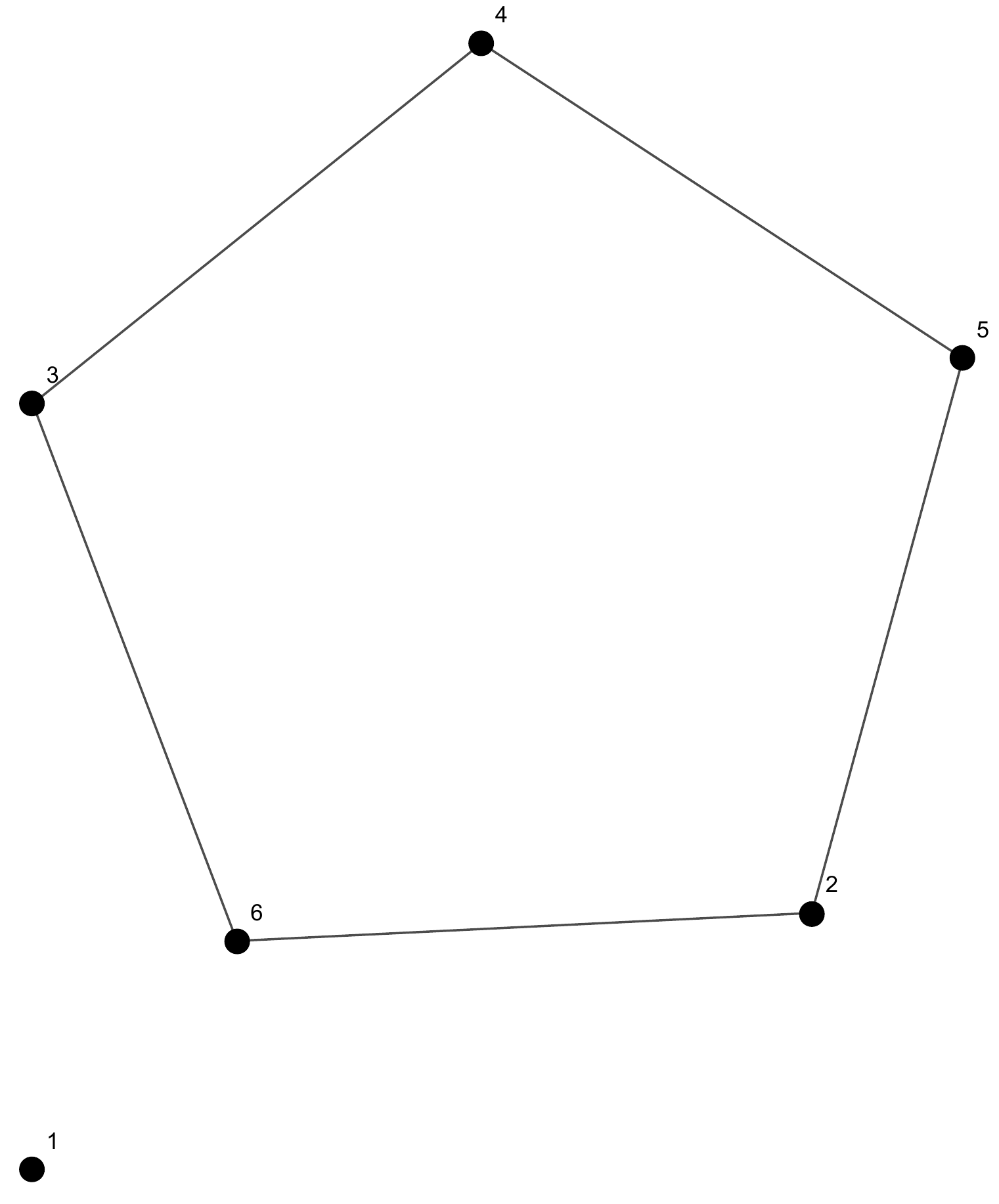}
\caption{The graph that has \(A_6-\mathbbm{1}_6\) as Seidel adjacency matrix. }
\label{fig:penta+1}
\end{figure}One can check that \(A_6-\mathbbm{1}_6\) is the Seidel adjacency matrix of the graph depicted in Figure \ref{fig:penta+1}. 
We abbreviate
\[\Omega:=\big\{ S : S \textrm{ is a principal submatrix of } A_6-\mathbbm{1}_6\big\}\cup \big\{ R_{2n+1}-\mathbbm{1}_{2n+1} : n\geq 1 \big\}.\]
In \cite{frankl1984exact},  Frankl and F{\"u}redi showed that each non-empty \(K_4\)-free two-graph belongs to the set
\[\delta(\Omega) \cup \big\{ \delta(B) : \textrm{ B  is a blow-up of a matrix in } \Omega \big\}.\]

\paragraph{\sffamily How to maximize the first \(n\) eigenvalues of \(AD\)?}
Given a matrix \(A\in \mathcal{A}_d\), we denote by \(\Stab(A)\) the set
\[\big\{ Q\in O_{d}(\mathbb{Z}) : A=Q A Q^t \big\}.\]
We use \(O_{d}(\mathbb{Z})\) to denote the group of orthogonal \(d\times d\)-matrices with integer entries. 
Every \(Q\in  \Stab(A) \) has a unique decomposition \(Q=PD\), where \(P\) is a permutation matrix and \(D\) is a diagonal matrix consisting only of \(1\)'s and \(-1\)'s. We write \(P_\tau:=P\) if the permutation matrix \(P\) is associated to the permutation \(\tau\), that is, \(P_{ij}=(e_{\tau(i)})_j\). 
The group  \(\Stab(A)\) acts on \(\left\{1, \ldots, d\right\}\) via 
\[(P_\tau D, k) \mapsto \tau(k).\]
Two Seidel adjacency matrices \(S_f\) and \(S_g\) are called \textit{switching equivalent} if \(\delta(f)=\delta(g)\). This gives rise to an equivalence relation, equivalence classes are called switching classes. 
The lemma below tells us that the orbit decomposition of the action \(\Stab(A)\curvearrowright \left\{1, \ldots, d\right\}\) may be obtained by determining the switching class of every principal \((d-1)-\)dimensional submatrix of \(A\).

\begin{lemma}\label{lem:dijon}
Let \(A\in\mathcal{A}_d\) be a matrix, let \(1\leq i,j \leq d\) be two integers and for \(k=i,j\) let \(T_k\) denote the submatrix of \(A-\mathbbm{1}_d\) obtained by deleting the \(k\)-th column and the \(k\)-th row of \(A-\mathbbm{1}_d\). 

Then, the matrices \(T_i\) and \(T_j\) are switching equivalent if and only if the integers \(i\) and \(j\) lie in the same orbit under the action \(\Stab(A)\curvearrowright \left\{1, \ldots, d\right\}\).
\end{lemma}
\begin{proof}
This is a straightforward consequence of the definitions. 
\end{proof}
Let \(M\) be a diagonalizable \(d\times d\)-matrix over the real numbers. We set
\[\pi_n(M):=\sum_{k=1}^n \lambda_k(M)\]
for each integer \(1\leq n \leq d\). The following lemma simplifies the calculation of the maximum value of the function \(D\mapsto \pi_n(AD)\) if the action \(\Stab(A)\curvearrowright \left\{1, \ldots, d\right\}\) is transitive. 
\begin{lemma}\label{lem:preChristmas}
Let \(A\in \mathcal{A}_d\) be a matrix and let \(1< n \leq d\) be an integer.
If \(\Lambda\in \mathcal{D}_d\) is a invertible matrix such that
\[\pi_n(A\Lambda)=\max_{D\in \mathcal{D}_d} \pi_n(AD),\]
then 
\[Q^2\Lambda (Q^2)^t=\Lambda\]
for all \(Q\in \Stab(A)\).
In particular, if \(d\) is odd and the action \(\Stab(A)\curvearrowright \left\{1, \ldots, d\right\}\) is transitive, then \(\Lambda=\frac{1}{d}\mathbbm{1}_d\). 
\end{lemma}
\begin{proof}
For each \(Q\in  \Stab(A)\) we have
\[Q^t \sqrt{\Lambda}A\sqrt{\Lambda} Q=Q^t \sqrt{\Lambda}Q A Q^t \sqrt{\Lambda}Q=\sqrt{\Lambda_{Q^t}}A\sqrt{\Lambda_{Q}},\]
where \(\Lambda_Q:=Q\lambda Q^t\). Consequently,
\[\pi_n(\sqrt{\Lambda} A \sqrt{\Lambda})=\pi_n\left(\sqrt{\Lambda_{Q^t}}A\sqrt{\Lambda_{Q}}\right)=\pi_n\left(A\sqrt{\Lambda_{Q}}\sqrt{\Lambda_{Q^t}}\right).\]
We get
\[1 \leq \Trr\left(\sqrt{\Lambda_{Q}}\sqrt{\Lambda_{Q^t}}\right).\]
Via the Cauchy-Schwarz inequality, we deduce
\[\Trr\left(\sqrt{\Lambda_{Q}}\sqrt{\Lambda_{Q^t}}\right)\leq 1;\]
as a result, there exists a real number \(\alpha\geq 0\) such that
\[\Lambda_{Q}=\alpha \Lambda_{Q^t}.\]
Since \(\Trr(\Lambda_{Q})=\Trr(\Lambda_{Q^t})=1\), we get \(\alpha =1\) and thus
\[\Lambda_{Q}=\Lambda_{Q^t},\]
which is equivalent to
\begin{equation*}
\Lambda_{Q^2}=\Lambda. 
\end{equation*}
Now, suppose that \(d\) is odd and assume that the action \(\Stab(A)\curvearrowright \left\{1, \ldots, d\right\}\) is transitive. We claim that \(\Lambda=\frac{1}{d}\mathbbm{1}_d\).  The statement follows via elementary group theory. Indeed, let \(H\) denote the subgroup of \(\Stab(A)\) generated by the squares.
By basic algebra, \(H\) is normal and the action of \(\Stab(A) / H\) on the orbits of \(H\curvearrowright \left\{1, \ldots, d\right\}\) is transitive.
Since \(\abs{\Stab(A) / H}=2^k\) for some integer \(k\geq 0\), the action \(H\curvearrowright \left\{1, \ldots, d\right\}\)
has either one orbit or an even number of orbits. Because \(d\) is odd and the orbits of \(H\curvearrowright \left\{1, \ldots, d\right\}\) all have the same cardinality, we may conclude that \(H\curvearrowright \left\{1, \ldots, d\right\}\) is transitive. 
This completes the proof. 
\end{proof}
\paragraph{\sffamily Determination of \(\Pi_2\)}\label{par:determin}
In the following we retain the notation from Paragraph \ref{para:classi}. 
By the use of Theorem \ref{thm:finite}, Lemma \ref{lem:holyGrail} and the classification of all \(K_4\)-free two-graphs, we obtain
\[\Pi_2=\max\limits_{(A-\mathbbm{1})\in \Omega} \max\limits_{D\in \mathcal{D}_d} \pi_2(AD).\]
Clearly, all induced sub-graphs of \(\mathcal{T}_{2n+1}\) that are obtained by deleting one vertex are isomorphic (as two-graphs) to each other. 
Thus, via Lemma \ref{lem:dijon} and Lemma \ref{lem:preChristmas}, we get that
\[\max\limits_{D\in \mathcal{D}_d} \pi_2\left(R_{2n+1}D\right)=\pi_2\left(\tfrac{1}{2n+1}R_{2n+1}\right).\]
Moreover, if \(B\) is a principal submatrix of \(A_6\), then it is not hard to see that
\[\max\limits_{D\in \mathcal{D}_d} \pi_2(BD)\leq \max\big\{\pi_2\left(\tfrac{1}{5}R_{5}\right), \pi_2\left(\tfrac{1}{3}R_{3}\right) \big\};\]
thereby,
\[\Pi_2=\max_{n\geq 1} \, \pi_2\left(\tfrac{1}{2n+1}R_{2n+1}\right).\]
Thus, we are left to consider the eigenvalues of the matrices \(R_{2n+1}\) for \(n\geq 1\). 
Due to the following lemma it suffices to calculate the eigenvalues of \(R_3\). 
\begin{lemma}\label{eq:AlmoistDone}
Let \(n^\prime \geq n \geq 1\) be integers. It holds 
\begin{equation}\label{eq:fearOfTheDark}
\pi_2\left(\tfrac{1}{2n+1}R_{2n+1}\right) \geq \pi_2\left(\tfrac{1}{2n^\prime+1}R_{2n^\prime+1}\right)
\end{equation}
\end{lemma}
\begin{proof}
We abbreviate \(N:=2n+1\). Let \(R^\prime_{N}\) denote the \(2n\times 2n\)-matrix that is obtained from \(R_{N}\) by deleting the second row and second column. 
Clearly, \(R^\prime_{N}\) is a blow-up of \(R_{N-2}\); thus, via Lemma \ref{lem:holyGrail}, we obtain
\[\max_{D\in \mathcal{D}_{2n}} \pi_2(R^\prime_{N} D)\leq \pi_2(R_{N-2}).\]
If for all integers \(k\geq 1\)
\begin{equation}\label{eq:metallica}
\pi_2\left(R_{2k+1}\right)=2\lambda_1\left(R_{2k+1}\right),
\end{equation}
then 
\[\pi_2\left(R_{N}\tfrac{1}{N}\right)=2\lambda_1\left(R_{N}^\prime\tfrac{1}{N-1}\right)\tfrac{N-1}{N}\leq 2\lambda_1\left(R_{N-2}\tfrac{1}{N-2}\right)\tfrac{N-1}{N}\]
and thus \eqref{eq:fearOfTheDark} follows. We are left to show that \eqref{eq:metallica} holds.

Suppose that \(\lambda_1(R_N)\) has multiplicity one. Below, we show that this leads to a contradiction. 

Let \(x\in \R^N\) be an eigenvector of \(R_N\) associated to the eigenvalue \(\lambda_1(R_N)\). As we assume that  \(\lambda_1(R_N)\) has multiplicity one,
we get \(Qx=x\) or \(Qx=-x\) for each \(Q\in \Stab(R_N)\). We know that the action \(\Stab(R_N)\curvearrowright \left\{1, \ldots, N\right\}\) is transitive; thus
all entries of \(x\) differ only by a sign. Without loss of generality we may suppose the entries of \(x\) consist only of \(1\)'s and \(-1\)'s.
For each integer \(1\leq i \leq N\) let \(A_i\) denote the matrix that is obtained from \(R_N\) by replacing the \(i\)-th column with \(x\). 
Cramers rule tells us that
\[x_i\det(R_N)=\det(A_i)\]
for all \(1\leq i \leq N\). It is easy to see (via the definition of \(R_N\)) that for all \(n<i<N\): if \(x_{i-n+1}\) and \(x_{i-n}\) have the same sign, then \(\det(A_i)=0\).
But this is impossible; for that reason, for all \(n<i<N\) we have \(x_{i-n}=-x_{i-n+1}\). Similarly,
\[x_{i+n}=x_{i+n-1}\]
for all \(2<i\leq n+1\) and \(x_1=-x_{n+1}\), \(x_2=-x_N\). Thus, if we suppose that \(x_1=1\), then
\[x=(1,\underbrace{-1,1,-1,1,\ldots 1,-1}_{\text{\(n\) times}},\underbrace{1,-1,1,-1,\ldots,-1,1}_{\text{\(n\) times}}) \quad \textrm{ if } n \textrm{ is odd}\]
and 
\[x=(1,\underbrace{1,-1,1,-1,\ldots,1,-1}_{\text{\(n\) times}},\underbrace{1,-1,1,-1,\ldots,1,-1}_{\text{\(n\) times}}) \quad                     \textrm{ if } n \textrm{ is even}.\]
Therefore, if \(j\in \R^N\) denotes the all-ones vector we obtain
\[\langle x, j \rangle=1,\]
and consequently it holds that
\[
\lambda_1(R_N)=
\begin{cases}
-1 & \textrm{ if } n \textrm{ is odd} \\
1 &  \textrm{ if } n \textrm{ is even}. 
\end{cases}
\]
This is a contradiction, since \(\Trr(R_N)=N\) and we assume that \(\lambda_1(R_N)\) has multiplicity one.
Hence, we have shown that the eigenvalue \(\lambda_1(R_N)\) has multiplicity greater than or equal to two.
As a result, \eqref{eq:metallica} holds, which was left to show. This completes the proof. 
\end{proof}
Employing Lemma \ref{eq:AlmoistDone}, we get
\[\Pi_2=\pi_2\left(\tfrac{1}{3}R_{3}\right) =\frac{1}{3}\big(2\lambda_1(R_3)\big)=\frac{1}{3}\big(3-\lambda_3(R_3)\big)=\frac{4}{3},\]
as conjectured by Grünbaum. 
\newline 
\paragraph{\sffamily An illustrative example: \(\Pi(4,6)=\frac{5}{3}.\)}\label{para:graph}

In this paragraph, we show that  \(\Pi(4,6)=\frac{5}{3}.\) We hope that some of the tools that are developed here may also simplify the computation of other relative projection constants. 

From a result of Sokołowski, cf. \cite{doi:10.1080/01630563.2016.1266655}, we obtain
\(\frac{5}{3}\leq \Pi(4,6)\). Given \(A\in \mathcal{A}_6\), we let \(n_+(A)\)  (or \(n_-(A)\)) denote the number of positive (or negative) eigenvalues of \(A\) counted with multiplicity.
Since \(\Pi(3,6)<\frac{5}{3}\) and \(\Pi(5,6) < \frac{5}{3}\), we deduce with the help of Theorem \ref{thm:known} that
\begin{equation*}
\Pi(4,6)=\max\big\{ \pi_4(AD) : A\in\mathcal{A}_6, \, n_+(A)=4,\,  n_-(A)=2, \, D\in\mathcal{D}_6\big\}.
\end{equation*}
In \cite{bussemaker1981tables}, Bussemaker, Mathon and Seidel classified all two-graphs on six vertices. Using this classification, we get that
there are exactly three non-isomorphic two-graphs on six vertices with signature \((n_+,n_-)=(4,2)\), namely \(T_i:=[A_i], i=1,2,3,\) for 
\begin{equation*}
A_1:= \left(
\begin{array}{cccccc}
 1 & -1 & 1 & 1 & 1 & 1 \\
 -1 & 1 & 1 & 1 & 1 & 1 \\
 1 & 1 & 1 & 1 & 1 & -1 \\
 1 & 1 & 1 & 1 & -1 & 1 \\
 1 & 1 & 1 & -1 & 1 & -1 \\
 1 & 1 & -1 & 1 & -1 & 1 \\
\end{array}
\right),
\end{equation*}
\begin{equation*}
A_2:= \left(
\begin{array}{cccccc}
 1 & 1 & 1 & 1 & 1 & 1 \\
 1 & 1 & -1 & 1 & 1 & 1 \\
 1 & -1 & 1 & 1 & 1 & 1 \\
 1 & 1 & 1 & 1 & -1 & -1 \\
 1 & 1 & 1 & -1 & 1 & -1 \\
 1 & 1 & 1 & -1 & -1 & 1 \\
\end{array}
\right),
\end{equation*}
and
\begin{equation*}
A_3:= \left(
\begin{array}{cccccc}
 1 & 1 & 1 & 1 & 1 & -1 \\
 1 & 1 & 1 & 1 & -1 & 1 \\
 1 & 1 & 1 & -1 & 1 & 1 \\
 1 & 1 & -1 & 1 & 1 & 1 \\
 1 & -1 & 1 & 1 & 1 & 1 \\
 -1 & 1 & 1 & 1 & 1 & 1 \\
\end{array}
\right).
\end{equation*}
Therefore,
\begin{equation*}
\Pi(4,6)=\max_{i=1,2,3} M_i, \quad \textrm{ for } \, M_i:=\max_{D\in \mathcal{D}_6} \pi_4(A_i D).
\end{equation*}
In \cite{Lieb1991}, Lieb and Siedentop established that the function
\[D\in\mathcal{D}_d\mapsto \pi_n(AD)\]
is concave if \(A\) is invertible and \(n_+(A)=n\). This result and the following lemma simplifies the computation of \(\Pi(4,6)\). 
\begin{lemma}\label{lem:symme}
Let \(A\in \mathcal{A}_d\) be an invertible matrix that has \(n:=n_+(A)\) positive eigenvalues. 
If the supremum
\[M:=\sup\big\{ \pi_n(AD) : D\in\mathcal{D}_d \textrm{ is invertible } \big\}\]
is attained, then there exist an invertible matrix \(\Lambda\in \mathcal{D}_d\) such that \(\pi_n(A\Lambda)=M\) and 
\begin{equation}\label{eq:anna3}
Q\Lambda Q^t=\Lambda \quad \textrm{ for all } Q\in \Stab(A).
\end{equation}
\end{lemma}
\begin{proof}
We set
\[X=\big\{D\in\mathcal{D}_d : D \textrm{ \normalfont is invertible and } \pi_n(AD)=M \big\}.\]
The function \(\pi_n(AD)\) is concave if we restrict the source space to the non-singular matrices in \(\mathcal{D}_d\); accordingly, it is immediate that the subset \(X\subset \mathcal{D}_n\) is convex.  Let \(Q\in \Stab(A)\). For all \(D\in \mathcal{D}_d\) we obtain
\[\pi_n(AD)=\pi_n(Q^t A Q D)=\pi_n( Q Q^t A Q D Q^t)=\pi_n(A Q D Q^t).\]
for that reason, for all \(Q\in \Stab(A)\) the set \(X\) is invariant under conjugation with \(Q\). If we equip \(X\) with the euclidean distance \(\norm{\cdot}_2\), then every map \(x\in X\mapsto Q x Q^t\) is an isometry. 
Since \((X,\norm{\cdot}_2)\) is a bounded non-empty CAT(0) space, there is a matrix \(\Lambda \in \overline{X} \subset \mathcal{D}_d\) with the desired properties, cf. \cite[II. Corollary 2.8]{bridson2013metric}. 
This completes the proof.  
\end{proof}

To begin, we show that
\begin{equation}\label{eq:1week}
M_1 \leq \frac{5}{3}.
\end{equation}
The value \(M_2\) will be estimated afterwards. By drawing the underlying graph of \(A_1\) and employing Lemma \ref{lem:dijon}, we may deduce that the action \(\Stab(A_1)\curvearrowright \{1, \ldots, 6\}\) has orbit decomposition \(\{1,2\},\, \{3,4\}, \,\{5,6\}\); consequently, Lemma \ref{lem:symme} tells us that
\[M_1=\max\big\{ \pi_4\big( A_1 \Lambda_1(s,t,w) \big)  : (s,t,w)\in \Delta^2 \big\},\] 
where \(\Delta^2\subset \R^3\) is the 2-dimensional standard simplex and \(\Lambda_1(s,t,w):=\Diag(s/2,s/2,t/2,t/2,w/2,w/2)\). 

The characteristic polynomial \(p_1(x)\) of the matrix \(A_1 \Lambda_1(s,t,w)\) can be written as \(p_1(x)=q_1(x) q_2(x) q_3(x)\), for
\[q_1(x):=x-s,\]
\[q_2(x):=-x^2+w x +t w\] and
\[q_3(x):=x^3-t x^2-s(1-s) x+s t w. \]
Clearly, the roots of \(q_2(x)\) are
\[\frac{1}{2} \left(w\pm\sqrt{w^2+4 t w}\right).\]
As \(p_1(x)\) has four positive roots, we obtain that \(q_3(x)\) has two positive roots. 
Let \(x_1 \geq x_2 \geq  x_3\) denote the roots of \(q_3\). We need to bound \(x_1+x_2\) from above.
By the virtue of Vieta's formulas, the roots \(\xi_1\geq \xi_2 \geq  \xi_3\)  of the polynomial
\[h(x):=x^3-2 t  x^2+ \left(t^2-s (1-s)\right) x+s t^2\]
satisfy \((\xi_1, \xi_2, \xi_3)=(x_1+x_2, x_1+x_3, x_2+x_3)\).
Hence, in order to estimate \(x_1+x_2\) from above, it suffices to bound the largest root of the polynomial \(h\) from above.
In the subsequent lemma we use Taylor's Theorem to get an upper bound for the largest root of a cubic polynomial with three real roots. 
\begin{lemma}\label{lem:cubic}
Let \(p(x)=x^3+bx^2+c x+d\), with \(b,c,d\in \R\), be a polynomial. If \(p\) has three real roots \(x_1\geq x_2 \geq x_3\), 
then
\[x_1 \leq \frac{1}{3} \left(-b+\sqrt{3} \sqrt{-A}+\frac{C}{6 (-A)}\right),\]
for
\[A:=3 c-b^2,\, \quad\, -C:=2 b^3-9 b c+27 d.\]
\end{lemma}
\begin{proof}As all roots are real, the cubic formula tells us that
\[x_1=\frac{-b}{3}+\frac{2}{3}\sqrt{-A}\cos\left(\frac{1}{3}\arccos\left(\frac{C}{2(-A)^{\frac{3}{2}}}\right)\right).\]
For \(a>0\) we define the map \(h_a\colon [-2 a^{3},2 a^{3}]\to \R\) via 
\[h_a(x):=a \cos\left(\frac{1}{3}\arccos\left(\frac{x}{2 a^{3}}\right)\right).\]
From Taylor's Theorem for each \(x\in [-2 a^3, 2 a^3]\) it holds that
\[h_a(x)=\frac{\sqrt{3} a}{2}+\frac{\sin \left(\frac{1}{3}\arccos \left(\frac{c}{2 a^3}\right)\right)}{6 a^2 \sqrt{1-\frac{c^2}{4 a^6}}} x,\]
where \(c\) is a real number between \(0\) and \(x\). 
Using elementary analysis, we obtain  
\begin{equation*}
\begin{cases}
\sin\left(\frac{1}{3}\arccos(x)\right)\leq \frac{1}{2}\sqrt{1-x^2} & \quad\textrm{ for all } x\in [0,1] \\
-\sin\left(\frac{1}{3}\arccos(x)\right)\leq -\frac{1}{2}\sqrt{1-x^2} & \quad\textrm{ for all } x\in [-1,0].
\end{cases}
\end{equation*} 
Therefore, we may deduce that
\begin{equation}\label{eq:lazy}
\frac{\sin \left(\frac{1}{3} \arccos\left(\frac{c}{2 a^3}\right)\right)}{6 a^2 \sqrt{1-\frac{c^2}{4 a^6}}} x \leq \frac{x}{12 a^2}
\end{equation}
for all \(x\in [-2 a^3, 2 a^3]\). 
All things considered, we get
\[x_1\leq \frac{1}{3} \left(-b+\sqrt{3} \sqrt{-A}+\frac{C}{6 (-A)}\right),\]
as was to be shown. 
\end{proof}
Employing the splitting \(p_1(x)=q_1(x) q_2(x) q_3(x)\) and Lemma \ref{lem:cubic}, we infer that \(\pi_4\big(A_1\Lambda_1\big) \) is less than or equal to 
	\begin{equation*}
	\begin{split}
f(s,t,w):=s+\frac{1}{2} \left(w+\sqrt{w^2+4 t w}\right)+ \frac{1}{3} \left(2t+\sqrt{3} \sqrt{-A}+\frac{C}{6 (-A)}\right),
	\end{split}
	\end{equation*}
	where
	\[A:=-3 s + 3 s^2 - t^2,\quad \quad C:=-27 s t^2+18 (1-s) s t-2 t^3.\]
	By elementary analysis, 
	\[\frac{C}{6 (-A)} \leq s+\frac{w}{2};\]
	as a result, we obtain
	\[f(s,t,w)\leq \frac{2}{3}+\sqrt{\frac{w^2}{4}+ t w}+  \sqrt{ \frac{t^2}{3}+s(1-s)}.\]
	Clearly,
	\begin{equation}\label{eq:miracle}
	\sqrt{\frac{w^2}{4}+ t w}+  \sqrt{ \frac{t^2}{3}+s(1-s)}\leq \sqrt{\frac{w^2}{2}+ t w}+  \sqrt{ \frac{t^2}{2}+s(1-s)}.
	\end{equation}
	Via Lagrange multipliers, the maximal value over \(\Delta^2\) of the right hand side of \eqref{eq:miracle}
	is equal to 
	\[\sqrt{\frac{1}{8} \left(2-\sqrt{2}\right)^2+\frac{2-\sqrt{2}}{2 \sqrt{2}}}+\frac{1}{2}=1.\]
	So, 
	\[M_1\leq \frac{5}{3},\]
	as desired.

	Next, we proceed with exactly the same strategy that dealt with \(M_1\) to show that 
	\[M_2 \leq \frac{5}{3}.\]
	As before, one can verify that the action \(\Stab(A_2)\curvearrowright \{1, \ldots, 6\}\) has orbit decomposition \(\{1\},\, \{2,3\},\, \{4,5,6\}\). The characteristic polynomial of
	 \(A_2\Lambda_2(s,t,w)\), for  \[\Lambda_2(s,t,w):=\Diag(s,t/2,t/2,w/3,w/3,w/3),\] is given by \(p_2(x)=-\frac{1}{9} r_1(x) r_2(x) r_3(x)\), with
\[r_1(x):=t - x,\]
\[r_2(x):=(2 w - 3 x)^2,\] and
\[r_3(x):= x^3-\left( s-\frac{w}{3}\right) x^2-\left( s t+\frac{4}{3} s w+ t w\right) x- \frac{4}{3} s t w .\]
Since all the roots of \(r_1(x)\) and \(r_2(x)\) are positive, we see that \(r_3(x)\) has two negative roots. 
Thus, employing Lemma \ref{lem:cubic} and Lemma \ref{lem:symme}, we obtain that \(\pi_4\big(A_2\Lambda_2(s,t,w)\big) \) is less than or equal to 
	\begin{equation*}
	\begin{split}
g(s,t,w):=t+\frac{4}{3} w+ \frac{1}{3} \left(s-\frac{w}{3}+\sqrt{3} \sqrt{-A}+\frac{C}{6 (-A)}\right),
	\end{split}
	\end{equation*}
	where
	\begin{equation*}
	\begin{split}
	&A:=-3 s t-\frac{1}{9} (w-3 s)^2-4 s w-3 t w, \\
	& C:=2 s^3+9 s^2 t+10 s^2 w+42 s t w-\frac{10 s w^2}{3}-3 t w^2-\frac{2 w^3}{27}.
	\end{split}
	\end{equation*}
	By elementary analysis, it is possible to verify that
	\[\frac{C}{18 (-A)}\leq \frac{1}{4}s+\frac{1}{6} \left(9 \sqrt{13}-32\right) t;\]
	for that reason,
	\begin{equation*}
	g(s,t,w)\leq \frac{1}{3}+\frac{1}{4}s+\frac{1}{6} \left(9 \sqrt{13}-28\right)t+\frac{8}{9}w+ \sqrt{3 s^2+9 s t+10 s w+9 t w+\frac{w^2}{3}}.
	\end{equation*}
Via Lagrange multipliers, the maximal value over \(\Delta^2\) of the right hand side of the above
	is equal to \(\frac{5}{3}\) for
	\[s:=0,\,\,\,\, t:=\frac{2 \left(108 \sqrt{13}+325\right)}{3539}\,\,\,\, w:=-\frac{27 \left(8 \sqrt{13}-107\right)}{3539}.\]
Thereby, with the help of Lemma \ref{lem:symme} we have established that
\[M_2 \leq \frac{5}{3}.\]
Next, we show that \(M_3=\frac{5}{3}\). The action \(\Stab(A_3)\curvearrowright \{1, \ldots, 6\}\) is transitive, that is, has orbit decomposition \(\{1,2,4,5,6\}\).
Thus, by Lemma \ref{lem:symme}, we deduce that
\[M_3=\pi_4(\tfrac{1}{6} A_3).\]
Let \(J_2\) denote the \(2\times 2\)-dimensional all-ones matrix and let \(R_3\) denote the matrix introduced in Paragraph \ref{para:classi}. 
Since \(\pi_2\left(\tfrac{1}{3}R_{3}\right)=\frac{4}{3}\), 
\[\pi_2\left(\tfrac{1}{6} J_2\otimes R_3 \right)=\frac{4}{3},\]
where we use \(\otimes\) to denote the Kronecker product of matrices. 
Consequently, for \(\overline{J_2\otimes R_3}=2\mathbbm{1}_{6}-J_2\otimes R_3\), we get
\[\pi_{4}\left(\tfrac{1}{6} \, \overline{J_2\otimes R_3} \right)=2\,\frac{4}{6}+\pi_2\left(\tfrac{1}{6} \, J_k\otimes R_3 \right)-1=\frac{7}{3}-\frac{4}{6}=\frac{5}{3}.\]
Considering \(\overline{J_2\otimes R_3}=A_3\), we get \(M_3=\frac{5}{3}\) and finally
\[\Pi(4,6)=\frac{5}{3},\]
as claimed. 
\begin{remark}
\normalfont As \(\R^5\) admits a system of ten equiangular lines, cf.  \cite[p. 496]{LEMMENS1973494}, the general upper bound of  König, Lewis, and Lin, cf. \cite{konig1983finite}, tells us that
\[\Pi(5,10)=2.\]
To summarize, we obtain the sequence
\[\Pi(1,1)=1,\quad \Pi(2,3)=\frac{4}{3}, \quad \Pi(4,6)=\frac{5}{3}, \quad \Pi(5,10)=2,\]
which naturally leads to the open question: Are there integers \(n \leq d\) such that
\[\Pi(n,d)=\frac{7}{3}?\]
\end{remark}

\paragraph{\sffamily Acknowledgements:} I am thankful to Urs Lang who read earlier draft versions of this article and who stimulated several improvements. 
Moreover, I am grateful to Anna Bot for proofreading this paper.

\newpage 
\printbibliography
\noindent
\newline
\textsc{\small{Mathematik Departement, ETH Zürich, Rämistrasse 101, 8092 Zürich, Schweiz}}\\
\textit{E-mail adress:}{\textsf{ giuliano.basso@math.ethz.ch}}

\end{document}